\documentclass[12pt]{amsart}

 \usepackage{amssymb,amscd}
\usepackage{epsf, amsmath, amsfonts, amsthm, amssymb}  

\newtheorem{thm}{Theorem}
\newtheorem{lem}[thm]{Lemma}
\newtheorem{cor}[thm]{Corollary}
\newtheorem{prop}[thm]{Proposition}

\theoremstyle{definition}

\newtheorem{ex}[thm]{Example}

\newtheorem{remark}[thm]{Remark}

\newcommand{\N}{\mathbb{N}}

\newcommand{\R}{\mathbb{R}}

\newcommand{\cA}{\mathcal A}
\newcommand{\cB}{\mathcal B}
\newcommand{\cC}{\mathcal C}
\newcommand{\cD}{\mathcal D}
\newcommand{\cDb}{\overline{\mathcal D}}
\newcommand{\cF}{\mathcal F}
\newcommand{\cE}{\mathcal E}

\newcommand{\cJ}{\mathcal J}
\newcommand{\cL}{\mathcal L}

\newcommand{\cP}{\mathcal P}

\newcommand{\cS}{\mathcal S}

\newcommand{\gb}{\overline{g}}

\newcommand{\ft}{\tilde f}
\newcommand{\gt}{\tilde g}

\newcommand{\keq}{\!=\!}

\newcommand\kin{\!\in\!}
\newcommand{\ksubset}{\!\subset\!}
\newcommand{\ksubseteq}{\!\subseteq\!}

\newcommand{\kplus}{\!+\!}
\newcommand{\kminus}{\!-\!}

\newcommand{\supp}{{\rm supp}}

\renewcommand{\d}{\,d}
\newcommand{\ord}{\text{\rm ord}}
\newcommand{\suc}{\text{\rm succ}}
\newcommand{\pred}{\text{\rm pred}}

\newcommand{\vp}{\varepsilon}

\newcommand{\ie}{\textit{i.e.,}\ }

\begin{document}
\title[Boundedness of Threshold Operators in $L_1$]{On the boundedness of threshold operators in $L_1[0,1]$ with respect to the Haar basis}

\author[Dilworth]{S. J. Dilworth}
\address{Department of Mathematics\\ University of South Carolina\\
Columbia, SC 29208\\ U.S.A.} \email{dilworth@math.sc.edu}
\author[Gogyan]{S. Gogyan}
\address{Institute of Mathematics,\\ Armenian Academy of Sciences, 24b Marshal Baghramian ave.\\ 
Yerevan, 0019\\ Armenia }
\email{gogyan@instmath.sci.am}
\author[Kutzarova]{Denka Kutzarova}
\address{Institute of Mathematics, Bulgarian Academy of
Sciences, Sofia, Bulgaria.} \curraddr{Department of Mathematics, University of Illinois,
Urbana-Champaign, Urbana, IL 61801, U.S.A.} \email{denka@math.uiuc.edu}
\author[Schlumprecht]{Th. Schlumprecht}
\address{Department of Mathematics, Texas A \& M University\\
College Station, TX 78743, U.S.A.}
\email{schlump@math.tamu.edu}
\allowdisplaybreaks \thanks{The first author was supported by the National Science Foundation under Grant Number DMS--1361461. The third author was partially supported by
the Bulgarian National Science Fund under Grant DFNI/I02/10. 
The fourth author  was supported by the National Science Foundation under Grant Numbers DMS--1160633
 and DMS--1464713.
The first and third authors were partially supported by the Workshop in Analysis and Probability at Texas A\&M University in  2015.}
\subjclass[2000]{46B15, 46B22}
\maketitle
\begin{abstract} We prove a near-unconditionality property for the normalized Haar basis
of $L_1[0,1]$. \end{abstract} 
\section{Introduction}\label{S:1}

Let $(e_i)$ be a semi-normalized basis for a Banach space $X$. For a finite subset $A \subset \mathbb{N}$,
let $P_A(\sum_i a_i e_i) := \sum_{i \in A} a_i e_i$ denote the projection from $X$ onto the span of basis vectors indexed by $A$.
Recall that $(e_i)$ is an unconditional basis if there exists a constant $C$ such that, for all finite $A \subset \mathbb{N}$,  $\|P_A\| \le C$.

We say that $(e_i)$ is near-unconditional if for all $0<\delta\le1$ there exists a constant $C(\delta)$ such that for all $x = \sum a_i e_i$
satisfying the normalization condition $\sup |a_i| \le 1$, and for all finite $A \subseteq \{ i \colon |a_i| \ge \delta \}$, $$\|P_A(x)\|  \le C(\delta)\|x\|.$$  
Every unconditional basis is near-unconditional, and it is easy to check that a near-unconditional basis is unconditional if and only if $C(\delta)$ can be chosen to be  independent of $\delta$.

 It was proved in \cite {DKK} that  a basis is near-unconditional if and only if  the thresholding operators $\mathcal{G}_\delta (x) := \sum_{|a_i| \ge \delta} a_i e_i$ satisfy,  for some constant $C_1(\delta)$,
$$ \|\mathcal{G}_\delta (x)\| \le C_1(\delta)\|x\|,$$ 
and that the class of  near-unconditional bases 
strictly contains the important class of quasi-greedy bases, 
defined by Konyagin and Temlyakov \cite{KT} as the class of bases for which $C_1(\delta)$ may be chosen to be independent of $\delta$.

Elton \cite{E} proved that every semi-normalized weakly null sequence contains a subsequence which is a near-unconditional basis for its closed linear span. On the other hand,  Maurey and Rosenthal \cite{MR} gave an example of a semi-normalized weakly null sequence with no unconditional subsequence.
%
% It was proved in \cite {DKK} that the a basis is near unconditional if and only if the the thresholding operators $\mathcal{G}_\delta (x) := \sum_{|a_i| \ge \delta} a_i e_i$ satisfy 
%$$ \|\mathcal{G}_\delta (x)\| \le C_1(\delta)\|x\|$$
% near-unconditional bases th
%strictly contains the important class of quasi-greedy bases, 
%defined by Konyagin and Temlyakov \cite{KT} as the class of bases for which the approximations generated by the Thresholding Greedy Algorithm are uniformly bounded.

By a theorem of Paley \cite{P}, the Haar system is an unconditional basis of $L_p[0,1]$ for $1 < p < \infty$. For $p=1$, on the other hand,
 a well-known example  (see e.g. \cite{LT}) shows
that the normalized Haar basis is not unconditional. The same example, which we now recall, also shows that the Haar basis fails to be near-unconditional.

Define $h_0=1_{[0,1]}$, and for $k\in\N$, set
   $$h^{(k)}_1=2^{k-1} \big(1_{[0,  2^{-k}) } -1_{[2^{-k}, 2^{1-k})}\big).$$
   Observe that for any $n\in\N$ 
   \begin{align*} 
   \Big\|h_0+\sum_{k=1}^{2n} h^{(k)}_1\Big\|=1,
   \intertext{ and for some  constant $c>0$}
    \Big\|h_0+\sum_{k=1}^{2n} h^{(2k)}_1\Big\|>cn.
   \end{align*}
So, setting $f_n = h_0+\sum_{k=1}^{2n} h^{(k)}_1$, and $A_n = \{0,2,4,\dots,2n\}$, we have $\|f_n\|=1$ and $\|P_{A_n}(f_n)\| \ge cn$,
which witnesses the failure of near-unconditionality with $\delta=1$.
In this example the nonzero coefficients of $f_n$ are equal and they  lie along the left  branch of the Haar system. Our 
main result shows that in a certain  sense every example of the failure of near-unconditionality must be of this type. 

We state our main result precisely below but the idea is as follows. Suppose  that
the Haar coefficients of  $f \in L_1[0,1]$, $\delta$, and $A$ are as stated in the definition of near-unconditionality. We show that
there is an enlargement $B \supseteq A$ such that $\|P_B(f)\| \le C(\delta)\|f\|$ and we provide an explicit construction of $B$. Roughly speaking,  the `added' coefficients in $B\setminus A$ are  those which lie along a segment of a branch of the Haar system such that the coefficient of the maximal element of the segment (with respect to the usual tree ordering)  belongs to $A$ and all the coefficients of $f$ along the  segment are approximately equal to each other (to within some prescribed multiplicative factor of $1+\varepsilon$).  For $f_n$ and the sets $A_n$, the enlargements are 
$B_n=\{0,1,2,\dots,2n\}$, and so $P_{B_n}(f_n)= f_n$, which renders the example harmless. Here the enlargement is as large as possible. The interest of our result, however,  resides in the fact that,  for certain $f$ and $A$,  the enlargement 
will often be trivial, i.e., $B=A$, or quite small.

The normalized Haar basis is not a quasi-greedy basis of $L_1[0,1]$, i.e., the Thresholding Greedy Algorithm fails to converge for certain initial vectors.  In a remarkable  paper  \cite{Go}  Gogyan exhibited a weak thresholding algorithm  which produces uniformly bounded approximants converging
 to $f$ for all $f \in L_1[0,1]$. The proof of our main theorem uses  results and techniques from \cite{Go}. We have chosen to reprove some of these results to achieve what we hope is a self-contained and accessible presentation. 

\section{Notation and basic facts}\label{S:2}
We denote the dyadic subintervals of $[0,1]$ by $\cD$, and put $\cDb=\cD\cup\{[0,2]\}$. We think of $\cD$ and $\cDb$ being partially ordered by ``$\subset$''.
We denote by $I^+$ and $I^-$ the left  and the right half subinterval of $I\in \cD$, respectively. 
 $I^+$ and $I^-$ are then the {\em direct successors} of $I$, while the set
 $$\suc(I)=\{J\in \cD :J\subsetneq I\}$$
 is called  the {\em successors of $I\in \cD$}. The {\em predecessors of an }$I\in\cDb$ is the set   %\marginpar{Steve: changed `are' to `is'}
 $$\pred(I)= \{J\in \cDb :J\supsetneq I\}.$$
 It follows that the $\pred(I)$ is a linearly ordered set. If $I\subset J$ are in $\cDb$ we put
 $$[I,J]=\{ K\in\cDb: I\subseteq K\subseteq J\}.$$

 Let $\cS\subset \cDb$ be finite and not empty. Then $\cS$ contains elements $I$ which are {\em minimal in $\cS$}, \ie there is no $J\in \cS$ for which $J\subsetneq I$.
 We put in this case
 $$\cS'=\cS\setminus \{ I\in\cS: I \text{ is minimal in }\cS\}.$$
 Inductively we define $S^{(n)}$ for $n\in\N_0$, by $S^{(0)}=\cS$, and, assuming $S^{(n)}$ has been defined, we put  $\cS^{(n+1)}= \big(S^{(n)}\big)'$.
 Since $\cS$ was assumed to be finite, there is an $n\in \N$, for which $\cS^{(n)}=\emptyset$, and we define {\em the order of $\cS$} by
 $$\ord(\cS)=\min\{ n\in\N: \cS^{(n)}=\emptyset\}-1=\max\{ n\in\N: \cS^{(n)}\ne \emptyset\}$$
 and for $I\in \cS$ we define the {\em order of $I$ in $\cS$} to be the (unique) natural number $m\in[0,\ord(\cS)]$, for which $I\in \cS^{(m)}\setminus \cS^{(m+1)}$, and we denote 
 it by $\ord(I,\cS)$.

$(h_I:I\in \cDb)$ denotes the  {\em  $L_1$-normalized Haar basis,} \ie
$$h_{[0,2]}=1_{[0,1]} \text{ and }h_I=2^{n} (1_{I^+}-1_{I^{-}}), \text{ if $I\in \cD$, with $m(I)=2^{-n}$,}$$ 
$m$ denoting  the Lebesgues measure. 
If $f\in L_1[0,1]$ we denote the coefficients of $f$ with respect to $(h_I)$ by $c_I(f)$, and thus 
\begin{equation}\label{E:2.1} f=\sum_{I\in \cD}  c_I(f) h_I\text{ for $f\in L_1[0,1]$.}\end{equation}
From the normalization of $(h_I)$ it follows that
\begin{equation}\label{E:2.2}
c_I(f)=\int_{I^+} f\d x-\int_{I^-} f\d x.
\end{equation}
For $f\in L_1[0,1]$  the {\em support of $f$ with respect to the Haar basis} is the set
$$\supp_H(f)=\{ I\in\cDb: c_I(f)\not=0\}.$$

We will use the following easy inequalities for $f\in L_1[0,1]$ and $I,J\in \cDb$, with $J\subseteq I$, \ie 
\begin{equation}\label{E:2.3}
\| f\|_I:= \int_I |f|\d x\ge  \int_J |f|\d x\ge\Bigg|\int_{J^+} f\d x-\int_{J^-} f\d x\Bigg|=|c_J(f)|.
\end{equation}

For  a finite set  $\cS\subset \cDb$ we denote by $P_{\cS}$ the canonical projection from $L_1[0,1]$ onto the span of $(h_I:I\in\cS)$:
$$P_{\cS}: L_1\to L_1,\,\, f\mapsto \sum_{I\in \cS} c_I(f) h_I.$$
If $\cS$ is cofinite $P_{\cS}$ is defined by $\text{Id}- P_{\cDb\setminus \cS}$. We will use the fact
that   the Haar system is a  monotone basis with respect to any order, which is consistent with the partial order ``$\subset$''. It follows therefore that 
the projections
 $$  S_I: L_1[0,1]\to L_1[0,1],\,\, f\mapsto f-\sum_{J\subseteq I} c_J(f) h_J,$$
 are bounded linear projections  with $\|\cS_I\|\le 1$, for all $I\in \cDb$. Moreover we observe that 
  \begin{align}\label{E:2.4}
 \|S_I(f)\|_I&= 
\Big\|\sum_{J\in \pred(I)} c_J(f) h_J\Big\|_I \\
            &\le \sum_{J\in \pred(I)} |c_J(f)| \|h_J\|_I \notag\\
& = \sum_{J\in \pred(I)} |c_J(f)| {\scriptstyle \frac{m(I)}{m(J)}}  \le \sup_{J\in\pred(I)} |c_J(f)|.\notag 
 \end{align}

For $f\in L_1[0,1]$,   $\vp>0$, and $\cA\subset \supp_H(x)$     we define
\begin{equation} \label{E:2.5}\cA_\vp=\Bigg\{ J\in \cDb: \exists I\kin \cA,\,\, I\ksubseteq J\text{ and }\Big|\frac{ c_{K}(f)-c_{I}(f)}{c_I(f)}\Big| <\vp, \text{ for all $K\in [I,J]$}\Bigg\}.
\end{equation}
Since $\cA_\vp$ depends on $\vp$ and the family $\big(c_I(f):I\in\cDb\big)$, we also write $\cA_{\vp}(f) $ instead of only $\cA_\vp$
to emphasize the dependence on $f$.

We are now ready to state our main result;

\begin{thm}\label{T:2.2} There is a universal constant $C$  so that  for $f\in L_1$, $\delta,\vp>0$  and 
$$\cA\subset \big\{ I\in \cDb: |c_I(f)|\ge\delta\big\}, $$
there is an  $\cE\subset \cDb$, with $\cA\subset  \cE\subset \cA_{\vp}(f)$, so that 
\begin{equation}\label{E:2.2.10}
\big\|P_\cE(f)\big\|\le C\frac{\log^2(1/\delta)}{\vp^2}\|f\|.
\end{equation}
\end{thm}
\begin{remark} The proof of Theorem~\ref{T:2.2} yields an explicit, albeit laborious,  description of $\cE$.
\end{remark}
\section{Proof of the main Result}\label{S:3}

We will first state and prove several Lemmas.

\begin{lem}\label{L:3.1} Let $f\in L_1[0,1]$ and  let $K,J$ and $I$  be elements of $\cDb$,  and  assume that 
$K$ is a direct successor of $J$, which is a  direct successor of $I$.
Then 
\begin{equation*}
\|f\|_{I\setminus K}\ge \Big| \frac{|c_I(f)|- |c_J(f)|}2\Big|.
\end{equation*}
\end{lem}
\begin{proof}
We first note that from the monotonicity property of the Haar basis we deduce
that 
\begin{align}\label{E:3.1.1}
\| f\|_{I\setminus J} &= \Big\| \sum_{L\in \pred(I\setminus J)} c_L(f) h_L|_{I\setminus J} + c_{I\setminus J}(f) h_{I\setminus J} + \sum_{L\in \suc(I\setminus J)}c_L(f) h_L\Big\|\\
&\ge \Big\| \sum_{L\in \pred(I\setminus J)} c_L(f) h_L|_{I\setminus J}  \Big\|= \| S_{I\setminus J}(f)\|_{I\setminus J}\notag
\end{align}
and similarly we obtain
\begin{equation}\label{E:3.1.2}
\| f\|_{J\setminus K}\ge \| S_{J\setminus K}(f)\|_{J\setminus K}.
\end{equation}
$S_I(f)$ takes a constant value $H$ on $I$. Denote by $a$ the value of $c_I(f) h_I$ on $I\setminus J$ and 
 denote by $b$ the value of $c_J(f) h_J$ on $J\setminus K$, and let $\delta=m(I)$. Then we compute
 
 \begin{align}\label{E:3.1.3}
\| f\|_{I\setminus K} &= \| f\|_{I\setminus J} +\| f\|_{J\setminus K}\\
&\ge  \| S_{I\setminus J}(f)\|_{I\setminus J}+  \| S_{J\setminus K}(f)\|_{J\setminus K}\notag\\
&=\|S_I(f)\kplus c_I(f) h_I\|_{I\setminus J}+ \|S_I(f)\kplus c_I(f) h_I\kplus c_J(f) h_J\|_{J\setminus K}\notag \\
&=\frac{\delta}2 |H+a| +\frac{\delta}4 |H-a+b|\notag\\
&\ge \frac{\delta}4 |H+a| +\frac{\delta}4 |-H+a-b|\notag\\
&\ge \frac{\delta}4 |2a-b|\ge  \frac{\delta}4 \big|2|a|-|b|\big|.\notag
\end{align}
Our claim follows then  if we notice that $|a|\delta = |c_I(f)|$ and $\delta |b|= 2|c_J(f)|$.
\end{proof}
We iterate Lemma \ref{L:3.1} to obtain the following result.
\begin{lem}\label{L:3.2} Let $f\in L_1[0,1]$ and  let $K,J$ and $I$  be elements of $\cDb$,  and  assume that 
$K$ is a  successor of $J$, which is a successor of $I$.
Then 
$$\|f\|_{I\setminus K}\ge \Big| \frac{|c_I(f)|- |c_J(f)|}4\Big| .$$
\end{lem}
\begin{proof} First we can, without loss of generality,  assume that $K$ is a direct successor of $J$,
we write $[K,I]$ as $[K,I]=\{ I_{n+1},I_n, I_{n-1}, \ldots I_0\}$, with 
$K=I_{n+1}\subsetneq I_n=J\subsetneq  I_{n-1}\subsetneq  \ldots I_0=I$, so that $I_{m+1}$ is a direct successor of $I_m$ for $m=0,1,2,\ldots,n$.
 From Lemma  \ref{L:3.1}  we obtain
 \begin{align*}
 \| f\|_{I\setminus K}&=\sum_{j=0}^n \|f\|_{I_j\setminus I_{j+1}}\\
 &\ge \frac12 \big( \|f\|_{I_0\setminus I_{1}}+ \|f\|_{I_1\setminus I_{2}}\big)+ \frac12 \big(  \|f\|_{I_1\setminus I_{2}}+\|f\|_{I_2\setminus I_{3}} \big)\\
   &\qquad\qquad\qquad\qquad\qquad\ldots+\frac12 \big(  \|f\|_{I_{n-1}\setminus I_{n}}+\|f\|_{I_n\setminus I_{n+1}} \big)\\
   &=\frac12\sum_{j=0}^{n-1} \|f\|_{I_j\setminus I_{j+2}}\\
   &\ge \sum_{j=0}^{n-1}\Big| \frac{|c_{I_j}(f)|- |c_{I_{j+1}(f)}|}4\Big|\ge \Big|   \frac{|c_{I}(f)|- |c_J(f)||}4  \Big|
 \end{align*}
 which finishes the proof of our assertion.
 \end{proof}
 \begin{lem}\label{L:3.3}  Assume that $f, g\in L_1[0,1]$  and $\cF\subset \cDb$ and that the  following properties hold  for some $\alpha,\vp\in(0,1)$ 
   \begin{enumerate}
  \item[a)] $\supp_H(f)\cap\supp_H(g)=\emptyset$,
  \item[b)] For every $I\in \cF$ there is a $J\in \suc(I)$, so that 
  \begin{align*}
  &[J,I]\cap \cF= \{I\}\\
  &|c_J(f)|\ge \alpha\\
  &|c_I(g)-c_J(f)|\ge \vp |c_J(f)|  .
  \end{align*} 
  \end{enumerate}
  Then
  \begin{equation}\label{E:3.3.0} \| f+g\|\ge \frac{\alpha\vp}6 |\cF|.\end{equation}  \end{lem}
 In order to prove Lemma \ref{L:3.3} we will first show the following observation.
 \begin{prop}\label{P:3.3a} Let $\cF\subset \cDb$, and define the following partition of $\cF$ into sets $\cF_0$, $\cF_1$ and $\cF_2$
   \begin{align*}
  &\cF_0=\{ I\in \cF: I \text{ is minimal in $\cF$}\}=  \{ I\in \cF: \suc(I)\cap \cF=\emptyset\}, \\
  &\cF_1= \{ I\in \cF: \suc(I)\cap \cF \text{ has exactly one maximal element} \}, \text{ and } \\
  &\cF_2= \{ I\in \cF: \suc(I)\cap \cF \text{ has  at least two maximal element} \} .
  \end{align*} 
 Then 
  \begin{equation}\label{E:3.3a.1}
  |\cF_2|< |\cF_0|.
  \end{equation}
 \end{prop}
 \begin{proof}
  In order to verify \eqref{E:3.3a.1}  we first  show for $I\in\cF_2$ that
     \begin{equation}\label{E:3.3a.2}
\big|\{ J\in\cF_2: J\subseteq I\}\big|<\big|\{J\in\cF_0: J\subset I\}\big| .
  \end{equation}
  Assuming that \eqref{E:3.3a.2}   is true for all $I\in \cF_2$, we let $I_1,I_2,\ldots I_l$ be the maximal elements of $\cF_2$. Since the $I_j$'s are pairwise disjoint, observe that
  \begin{align*}
  |\cF_2|=\sum_{j=1}^l \big|\{ I\in\cF_2: I\subseteq I_j\}\big|< \sum_{j=1}^l \big|\{J\in\cF_0: J\subset I_j \}\big|\le |\cF_0|.
  \end{align*}
  We  now prove \eqref{E:3.3a.2} by induction on $n=\big|\{ J\in\cF_2: J\subseteq I\}\big|$. If $n=1$ then $I$ must have at least two successor, say  $J_1$ and $J_2$ in $\cF$
  which are  incomparable, and thus there are elements $I_1,I_2\in \cF_0$ so that $I_1\subset J_1$ and $I_2\subset J_2$.
   Assume that our claim is true for $n$, and assume that  
$\big|\{ J\in\cF_2: J\subseteq I\}\big|=n+1\ge 2$. We denote the maximal elements of $\{ J\in \cF_2 : J\subsetneq I  \}$, by $I_1, I_2,\ldots I_m$. 
Either $m\ge 2$, then it follows from the 
induction hypothesis, and the fact that  $I_1$,$I_2,\ldots I_m$ are incomparable, that   
\begin{align*}
|\{ J\in \cF_2 : J\subseteq I \}|&=1+ \sum_{j=1}^m \big|\{ J\in \cF_2 : J\subseteq I_j  \}\big|\\
&\le 1+ \sum_{j=1}^m \big( \big| \{ J\in \cF_0 : J\subseteq I_j  \}\big|-1\big)\\
&\le \big|\{J \in\cF_0 : J\subseteq  I\}\big|-1<  \big|\{J \in\cF_0 : J\subseteq  I\}\big|.
\end{align*}
Or  $m=1$, and if $\tilde I$ is the only maximal element of  $\{ J\in \cF_2 : J\subsetneq I  \}$, then by the definition of $\cF_2$ there must be a $J_0\in \cF_0$ with
 $J_0\subset I\setminus \tilde I$, and we deduce from our induction hypothesis that
\begin{align*}
|\{ J\in \cF_2 : J\subseteq I \}|&= 1+
\big|\{ J\in \cF_2 : J\subseteq\tilde  I \}\big| \\
&<1+|\{ J\in \cF_0 : J\subseteq\tilde  I \}| \le\big|\{ J\in \cF_0 : J\subseteq  I \}\big|,
\end{align*}
which finishes the proof of the induction step, and the proof of \eqref{E:3.3a.2}.
\end{proof}

 \begin{proof}[Proof of Lemma \ref{L:3.3}]
Assume now that $\alpha,\vp>0$ and $f,g\in L_1[0,1]$, and $\cF\subset \cD$  are given satisfying (a), (b). 
Let $\cF_0$,  $\cF_1$, and  $\cF_2$ the subsets of $\cF$ introduced in Proposition \ref{P:3.3a}.
We distinguish between two cases.

 \noindent{\bf Case 1.} $| \cF_0|\ge \frac16 |\cF|$.
 
 Fix $I\in \cF_0$, and let $J\in\suc(I)$ be chosen so  that  condition (b) is satisfied. It follows then from condition  (a) and \eqref{E:2.3}
 $$\|f+g\|_{I}\ge \|f+g\|_{J}\ge |c_J(f)|\ge \alpha.$$
 Since all the elements in $\cF_0$ are disjoint it follows that 
  $$\|f+g\|\ge \sum_{I\in \cF_0} \|f+g\|_{I}\ge \big|\cF_0\big| \alpha\ge \big|\cF\big| \frac{\alpha}6.$$
 
  \noindent{\bf Case 2.} $| \cF_0| < \frac16 |\cF|$.

Applying \eqref{E:3.3a.1} we obtain that 
$$|\cF_1|=|\cF |-|\cF_0|-|\cF_2|>|\cF|-2|\cF_0|>\frac23 |\cF|.$$

 Fix $I\in \cF_1$, and let $J\in \suc(I)$ satisfy the conditions in (c), and let $\tilde I$,  be the unique maximal element of
 $\suc(I)\cap \cF$. It follows that $ J\subsetneq I$ and, since by condition (b) $\tilde I\not\in [J,I]$, we deduce that $J\not\subset \tilde I$ which implies that either 
 $\tilde I\subsetneq J$ or $\tilde I\cap J=\emptyset$. 
 In the first case we deduce from Lemma \ref{L:3.2}, and condition (b)
 that 
 \begin{align*}
 \|f+g\|_{I\setminus \tilde I}
 &\ge \Big|\frac{|c_ I(g)|- |c_J(f)|}4 \Big|
\ge \vp\frac{|c_{J}(f)|}4\ge \frac{\vp \alpha}4.
 \end{align*}  
 In the second case we deduce  from \eqref{E:2.3}  and condition (c) that 
  \begin{align*}
 \|f+g\|_{I\setminus \tilde I}\ge \|f+g\|_{J}\ge | c_{J}(f)|\ge  \alpha.
 \end{align*}
We conclude therefore   from the fact that the sets $I\setminus \tilde I$,  with $I\in \cF_1$,  are pairwise  disjoint and therefore
$$\|f+g\|\ge \sum_{I \in \cF_1} \|f+g\|_{I\setminus \tilde I} \ge |\cF_1| \frac{\vp\alpha}4
\ge \frac23|\cF| \frac{\vp\alpha}4 = \frac{\alpha\vp}6 |\cF|. $$
  \end{proof}
  In order to formulate our next step we introduce the following {\em Symmetrization Operators} $\cL_1$ and $\cL_2$. 
  For that assume that $f\in L_1[0,1] $  and $I\in \cD$. We  define the following two functions $\cL_1(f,I)$ and $\cL_2(f,I)$ in $L_1[0,1]$.
  For $\xi\in[0,1]$ we put
  \begin{align*}
  &\cL_1(f,I)(\xi)=\begin{cases} f(\xi)                                 &\text{if  $\xi\not\in I^-$}\\
                                                  f\big(\xi- \frac{m(I)}2) &\text{if   $\xi\in I^-$}\end{cases}  \\                                           
                   &\cL_2(f,I)(\xi)=\begin{cases} f(\xi)                                 &\text{if  $\xi\not\in  I^+$}\\ 
                                                  f\big(\xi+\frac{m(I)}2\big) &\text{if   $\xi\in  I^+$}\end{cases}                              
  \end{align*}
 Note that $\cL_1(f,I)$ restricted to $I^-$ is  a shift of $f$ restricted to $I^+$, and vice versa 
  $\cL_2(f,I)$ restricted to $I^+$ is  a shift of $f$ restricted to $I^-$.

We will use this symmetrization only for $f\in L_1[0,1]$ and $I\in \cD$, for which $c_I(f)=0$, We observe in that case  that  letting $f'=\cL_1(f,I)$  or  $f'=\cL_2(f,I)$, and   any $J\in \cD$ 
\begin{equation}\label{E:3.1}
c_J(f')=\begin{cases}   c_J(f) &\text{if $J\supsetneq I$} \text{ (here we use that $c_I(f)=0$)}\\
 c_J(f) &\text{if $J\cap I=\emptyset$}\\
                                         0           &\text{if $J= I$}\\
                                         c_J(f)  &\text{if $J\subset I^+$ and $f'=L_1(f,I)$, or }\\
                                                     &\text{if $J\subset I^-$ and $f'=L_2(f,I)$, }\\
                                         c_{J-m(I)/2}(f)              &\text{if $J\subset I^-$ and $f'=L_1(f,I)$ }\\
 c_{J+m(I)/2}(f)              &\text{if $J\subset I^+$ and $f'=L_2(f,I)$. }
   \end{cases}
   \end{equation}
   Moreover it follows  that 
   \begin{align}\label{E:3.2}
   \big\|L_1(f,I)\big\|& = \|f\|+\Delta(f,I) \text{ and }\big\|L_2(f,I)\big\| = \|f\|-\Delta(f,I), \\
   & \text{ with }\Delta(f,I)=\|f\|_{I^+}-\|f\|_{I^-}. \notag
   \end{align}
   \begin{lem}\label{L:3.4}  Assume that $f,g\in L_1[0,1]$, so that  $\supp_H(f), \supp_H(g)\subset \cD$  and $I\in \cD$ are given.
Suppose that $c_I(f) = c_I(g)=0$ and that  the following properties hold: 
   \begin{enumerate}
   \item[a)] $\supp_H(f)\cap \supp_H(g)=\emptyset$,
   \item[b)] $\|f\|>0$, and thus, by (a), also $\|f+g\|>0$.
   \end{enumerate}
   Then for either  $f'=L_1(f,I)$  and  $g' = L_1(g,I)$, or $f'=L_2(f,I)$   and  $g' = L_2(g,I)$ it follows that 
    \begin{align}\label{E:3.4.1} 
    &\supp_H(f')\cap\supp_H(g')=\emptyset \text{ and }I\not\in \supp_H(f')\cup\supp_H(g');\\
       \label{E:3.4.3}&\text{for any $J\in\cD$ it follows that}\\
     &\qquad c_J(f')\kin \{ c_{I}(f'):I\kin\supp(f)\}\cup\{0\} \text{ and }\notag\\
     &\qquad c_J(g')\kin \{ c_{I}(g'):I\kin\supp(g)\}\cup\{0\},\notag\\
    \label{E:3.4.2}
    &\|f'+g'\|>0 \text{ and } \frac{\|f\|}{\|f+g\|}\le \frac{\|f'\|}{\|f'+g'\|}.
    \end{align}
   \end{lem}
   \begin{proof} It follows immediately from \eqref{E:3.1} that  \eqref{E:3.4.3} and  \eqref{E:3.4.1}  are satisfied for either  of the possible choices of $f'$ and $g'$.
   
   To satisfy  \eqref{E:3.4.2} we will first consider the case that $(f+g)|_{[0,1]\setminus I^+}\equiv 0$.
   In this case it follows from (a) that $c_J(f)=c_J(g)=0$ for all $J\in \cD$, with $J\subseteq [0,1]\setminus I^+$,
   and thus by (b)  $f|_{I^+}\not\equiv 0$ and $(g+f)|_{I^+}\not\equiv 0$.
   If we choose $f'=L_1(f,I)$ and $g'=L_1(g,I)$ we obtain that $\|f'+g'\|>0$,
   $\|f'\|=2\|f\|$ and $\|f'+g'\|=2\|f+g\|$.

   A similar argument can be made if $(f+g)|_{[0,1]\setminus I^-}\equiv 0$.

   If neither of the two previously discussed cases occurs we conclude that
   \begin{align*}
   \|f\kplus g\|&= \|f\kplus g\|_{[0,1]\setminus I} +\|f\kplus g\|_{I^+} +   \|f\kplus g\|_{I^-} >\|f\kplus g\|_{I^+} -   \|f\kplus g\|_{I^-} \intertext{ and }
    \|f\kplus g\|&= \|f\kplus g\|_{[0,1]\setminus I} +\|f\kplus g\|_{I^+} +   \|f\kplus g\|_{I^-} >\|f\kplus g\|_{I^-} -   \|f\kplus g\|_{I^+}
\end{align*}
    which implies by \eqref{E:3.2} that in either of the two  possible choices for $f'$ and $g'$ it follows that $\|f'+g'\|>0$.

Finally,  if  $\Delta(f,I)\|f+g\|\ge \Delta(f+g,I)\|f\|$, we choose $f'=L_1(f,I)$ and $g'=L_1(f,I)$  and note that 
since in this case we have 
$$(\|f\|+\Delta(f,I))\|f+g\|\ge (\|f+g\|+\Delta(f+g,I))\|f\|,$$
it follows   that 
$$ \frac{\|f'\|}{\|f'+g'\|}=  \frac{\|f\|+\Delta(f,I)}{\|f+g\|+\Delta(f+g,I)} \ge 
 \frac{\|f\|}{\|f+g\|}.$$
 If  $\Delta(f,I)\|f+g\|< \Delta(f+g)\|f\|$ and thus $-\Delta(f,I)\|f+g\|>-\Delta(f+g)\|f\|$,  we choose $f'=L_2(f,I)$ and $g'=L_2(f,I)$  and note that 
since in this case we have 
$$\big(\|f\|-\Delta(f,I)\big)\|f+g\|> \big(\|f+g\|-\Delta(f+g,I)\big)\|f\|,$$
 and it follows   that 
$$ \frac{\|f'\|}{\|f'+g'\|}=  \frac{\|f\|-\Delta(f,I)}{\|f+g\|-\Delta(f+g,I)} >
 \frac{\|f\|}{\|f+g\|},$$
  which finishes the verification of \eqref{E:3.4.2} and the proof of our claim. 
  \end{proof}

Assume now that $f$, $g\subset L_1[0,1]$, $\|f\|>0$, are such that 
$\supp_H(f)$ and $\supp_H(g)$ are finite and disjoint subsets of $\cD$.  We also assume that 

\begin{align}\label{E:3.4b} 
c_{[0,1]}(f)&=\int_{0}^{1/2} f(x)\,d x-\int_{1/2}^{1} f(x)\,d x=0 \text{ and }\\
c_{[0,1]}(g)&=\int_{0}^{1/2} g(x)\,d x-\int_{1/2}^{1} g(x)\,d x =0.\notag
\end{align} 

Define:
\begin{equation}\label{E:3.4c}
\cF(f)=\{ I\in \cD:  c_I(f)=0\text{ but  } c_{I^+}(f)\not=0 \text{ or }   c_{I^-}(f)\not=0\}
\end{equation}
and make the following assumption
\begin{equation}\label{E:3.4a} 
\supp_H(g)\cap \cF(f)=\emptyset.
\end{equation}

Let $\cF(f)= (I_i)_{i=1}^n$, where $m(I_1) \le m(I_2) \le \dots \le m(I_n)$. First `symmetrize' the pair  $(f,g)$ on $I_1$ to obtain
a pair $(f_1, g_1)$ satisfying $\|f\|/\|f+g\| \le \|f_1\|/\|f_1+g_1\|$. Note that $\cF(f_1) = \cF(f)$. Now symmetrize $(f_1,g_1)$
on $I_2$ to obtain $(f_2, g_2)$ satisfying $\|f_1\|/\|f_1+g_1\| \le \|f_2\|/\|f_2+g_2\|$. Note that  if $I \in  \cF(f_2)$ satisfies $m(I) < m(I_2)$, then $(f_2,g_2)$ on $I$ is a `copy' of $(f_1,g_1)$ on
$I_1$.  Hence, $f_2$ and $g_2$ are automatically symmetric on $I$. On the other hand, if $I \in \cF(f_2)$ satisfies 
$m(I) \ge m(I_2)$ then $I = I_j$ for some $j \ge 2$. Now symmetrize $(f_2,g_2)$ on $I_2$ to obtain $(f_3,g_3)$ satisfying 
$\|f_2\|/\|f_2+g_2\| \le \|f_3\|/\|f_3+g_3\|$. Note that  if $I \in  \cF(f_3)$ satisfies $m(I) < m(I_3)$ then $(f_3,g_3)$ on $I$ is a copy of  $(f_2,g_2)$ on
$I_1$ or $I_2$. Hence $f_3$ and $g_3$ are automatically symmetric on $I$. 
Continuing in this way, we finally obtain,
after symmetrizing on $I_n$, a pair $(f_n,g_n)$ such that $f_n$ and $g_n$ are symmetric on each $I \in \cF(f_n)$
and $\|f\|/\|f+g\| \le \|f_n\|/\|f_n+g_n\|$.

 Setting $\ft=f_n$ and $\gt= f_n$, 
  the following conditions  hold: 
   \begin{align}
   \label{E:3.5} &  c_{[0,2]}(\ft)=c_{[0,2]}(\gt) =c_{[0,1]}(\ft)=c_{[0,1]}(\gt)=0,\\
  \label{E:3.6} &\text{For all $I\in\cF(\ft)$ it follows that}\\
                      &\qquad c_{I}(\gt)=0\notag \text{ and }\\
                    &\qquad \ft(x)=\ft(x\kminus m(I^+))\text{ and } \gt(x)=\gt(x\kminus m(I^+)), \text{ if $x\kin I^-$,}\notag\\
   \label{E:3.7}    &\supp_H(\ft)\cap\supp_H(\gt)=\emptyset, \\
       \label{E:3.8}&\text{for any $J\in\cD$ it follows that}\\
     &\qquad c_J(\ft)\kin \{ c_{I}(f):I\kin\supp(f)\}\cup\{0\} \text{ and }\notag \\
     &\qquad c_J(\gt)\kin \{ c_{I}(g):I\kin\supp(g)\}\cup\{0\},\notag\\
    \label{E:3.9}
    &\|\ft+\gt\|>0 \text{ and } \frac{\|f\|}{\|f+g\|}\le \frac{\|\ft\|}{\|\ft+\gt\|}. 
   \end{align}

   \begin{lem}\label{L:3.5} Assume that $f,g\in L_1[0,1]$  are such that $\supp_H(f)$ and $\supp_H(g)$ are finite disjoint subsets of
    $\cD\setminus\{[0,1]\}$, and that 
    $\supp_H(g)\cap \cF(f)=\emptyset$, where $\cF(f)\subset \cD$, was defined above. 
     Assume moreover that for some $\alpha>0$, we have
     \begin{align}\label{E:3.5.1a}
    &\text{$|c_{J}(f)|\ge \alpha$, for all $J\in \supp_H(f)$, and}\\ 
    &\text{$|c_J(g)|\le 1$, for all $J \in  \supp_H(g)$.}\notag
    \end{align} 
    
  Then 
  \begin{equation} \label{E:3.5.1}
  \|f\|\le \Big(\frac5\alpha+1\Big) \|f+g\|.
  \end{equation}
   \end{lem}
   \begin{proof} Let $\ft$ and $\gt$ be the elements in $L_1[0,1]$ constructed from $f$ and $g$ as before  
   satisfying the conditions  \eqref{E:3.5}, \eqref{E:3.6}, \eqref{E:3.7}, \eqref{E:3.8}, and \eqref{E:3.9}.  
   Note also that \eqref{E:3.8} implies that $|c_{J}(\ft)|\ge \alpha$, for all $J\in \supp_H(\ft)$, and 
   $|c_J(\gt)|\le 1$, for all $J\in  \supp_H(\gt)$.
   
 By  \eqref{E:3.9} it  is enough to show  \eqref{E:3.5.1}  for $\ft$ and $\gt$ instead of $f$ and $g$.    
   
   We will  deduce our statement from the following
   
  \noindent {\bf Main Claim.}  For all  $I\in\cF(\ft)$ it follows that 
  \begin{equation}\label{E:3.5.2}
  \|\ft\|_I\le \Big(\frac5\alpha+1\Big) \|\ft+\gt\|_I -2\alpha -8.
  \end{equation}   
 Assuming the Main Claim we  can argue as follows. Using \eqref{E:3.5} it follows that out side of  $J=\bigcup_{I\in\cF(\ft)} I$ $\ft$ is vanishing. Thus,  we can choose disjoint   sets $I_1,I_2,\ldots I_n$, in 
   $\cF(\ft)$ so that $\ft$ vanishes outside of $\bigcup_{j=1}^n I_j $, and \eqref{E:3.5.2} yields 
   $$\|\ft\|=\sum_{j=1}^n \|\ft\|_{I_j}\le  \Big(\frac5\alpha+1\Big)   \sum_{j=1}^n\|\ft+\gt\|_{I_j}\le  \Big(\frac5\alpha+1\Big)  \|\ft+\gt\|,$$
   which proofs our wanted statement.

   In order to show the Main Claim let $I\in\cF(\ft)$ and denote $k=\ord(I, \cF(\ft))$. We will show the inequality \eqref{E:3.5.2} by induction for all $k$.
   
   First assume that $k=0$. From  \eqref{E:3.6} and  \eqref{E:2.3} we obtain
   \begin{equation}\label{E:3.5.3}
   \|\ft+\gt\|_I= 2 \|\ft+\gt\|_{I^+} \ge2 \big| c_{I^+}(\ft+\gt)\big|=\big|2 c_{I^+}(\ft)\big|\ge 2\alpha .
   \end{equation}
   Using \eqref{E:2.4}, \eqref{E:3.7} and \eqref{E:3.5.1} we obtain
    \begin{equation}\label{E:3.5.3a}
  \|S_I(\gt)\|_I\le 1.
     \end{equation} 
     From the definition of $\cF(\ft)$, and  the  assumption that  $\ord\big(I, \cF(\ft)\big)=0$,
        \eqref{E:3.7} we deduce that if $J\in \supp_H(\ft)$ with $J\subset I^+$ or
     $J\subset I^-$, then $[J,I^+]\subset \supp_H(\ft)$, or $[J,I^-]\subset \supp_H(\ft)$, respectively.
     But, using \eqref{E:3.7}, this implies that  if $J\in\supp_H(\gt)$, with $J\subset I$, then 
     $\suc(J)\cap \supp_H(\ft)=\emptyset$.
      We deduce therefore from the monotonicity properties of the Haar basis
     that
     \begin{equation}\label{E:3.5.4}
     \|\ft+\gt\|_I\ge  \Big\|f+\sum_{J\in\cD, J\not\subset I} c_J(\gt) h_J  \Big\|_I= \|\ft+S_I(\gt)\|_I.
     \end{equation}
     We therefore conclude
     \begin{align*}
     \Bigg(\frac5\alpha+2\Bigg)\|\ft+\gt\|_I&\ge    \Bigg(\frac5\alpha+1\Bigg)\|\ft+\gt\|_I+ \|\ft +S_I(\gt)\|_I \qquad\text{(by \eqref{E:3.5.4})}\\
      &\ge 10+2\alpha +\|\ft\|_I- \|S_I(\gt)\|\qquad\text{(by \eqref{E:3.5.3})}\\
      &\ge \|f\|_I+9+2\alpha \text{  (by \eqref{E:3.5.3a}),} 
     \end{align*}
     which proves our claim in the case that $\ord\big(I,\cF(\ft)\big)=0$.

     Assume that \eqref{E:3.5.2} holds  for all $I\in\cF$ with $\ord\big(I,\cF(\ft)\big)< k$, for some $k\in\N$, and assume that 
     $I\in \cF(\ft)$ with $\ord(I,\cF(\ft))=k$.   By the symmetry condition in \eqref{E:3.6} the number of elements $J$ of $\cF(\ft)$
     for which $J\subset I$ and $\ord(J, \cF)=k-1$ is even, half of them being subsets of $I^+$, the other  
      being subsets of $I^-$. We order therefore these sets into $J_1$, $J_2,\ldots J_{2s}$, for some $s\in\N$, with $J_i\subset I^+$
      and $J_{s+i}\subset I^-$, for $i=1,2,\ldots,s$.
      We note that the $J_i$, $i=1,2\ldots $, are pairwise disjoint and that all the $J\in\cF(\ft)$, with $F\subset I$, and $\ord(J,\cF(\ft))\le k-2$, are subset       of some of the $J_i$, $i=1,2\ldots 2s$.
      
From  our induction hypothesis we deduce that 
    \begin{equation}\label{E:3.5.5}
    \| \ft\|_{J_i}\le \Big(\frac5\alpha+2\Big)\|\ft+\gt\|_{J_i} -2\alpha-8\text{ for }i=1,2,\ldots, 2s. 
    \end{equation}
    
    We define  $D=I^+\setminus \bigcup_{i=1}^s J_i$ and
    \begin{align*}
    &\phi= S_{J_1}(S_{J_2}(\ldots S_{J_s}(\ft)\ldots ))= \sum_{J\in\cJ} c_J(\ft) h_J\\
    &\gamma= S_{J_1}(S_{J_2}(\ldots S_{J_s}(\gt)\ldots ))= \sum_{J\in\cJ} c_J(\gt) h_J
    \end{align*}
      with
      $$\cJ=\big\{ J\in \cD: \forall\,j\keq1,2\ldots,s \quad J\not\subset I_j\big\}.$$ 
      It follows that  
            \begin{align}
\label{E:3.5.6}   &\phi|_D=\ft|_D \text{ and }\gamma|_D=\gt|_D,
      \intertext{\eqref{E:2.4} implies that}
   \label{E:3.5.7}   &\|\gamma\|_{J_i}\le 1, \text{ for $i=1,2\ldots s$,}
    \intertext{and since for any $J\in\cD$, with $J\subseteq D$, 
for which $c_J(\phi)\not=0$, we have  $[J,I^+]\subset \supp_H(\phi)$
      (otherwise there would be an $K\in \cF(\ft)$ with $K\subset I^+$ and  $K\supsetneq J_i$, for some $i\in\{1,2\ldots s\}$, or $K\subset D$)
      it follows from the monotonicity property of the Haar system  and   \eqref{E:2.4} that }
         \label{E:3.5.8} 
        & \|\phi+\gamma\|_{I^+} \ge \|\phi+S_{I^+}(\gamma)\|_{I^+}\ge \|\phi\|_{I^+}-1.
      \end{align}       It follows that 
      \begin{align*}
      \|\phi+\gamma\|_{I^+\setminus D}&\le \|\phi\|_{I^+\setminus D}+  \|\gamma\|_{I^+\setminus D}\\
         &=\|\phi\|_{I^+\setminus D}+\sum_{i=1}^s \|\gamma\|_{J_i}\\
         &\le \|\phi\|_{I^+\setminus D} +s 
            \end{align*}
      and 
         \begin{align*}
       \|\ft+\gt\|_D&=\|\phi+\gamma\|_{D}\\ 
       &=\|\phi+\gamma\|_{I^+}  - \|\phi+\gamma\|_{I^+\setminus D} \\ 
      &\ge \|\phi\|_{I^+} -1 - \|\phi\|_{I^+\setminus D} -s =\|\phi\|_D-s-1.
      \end{align*}
      This implies together with \eqref{E:3.5.5} that
      \begin{align*}
      \|\ft\|_{I^+}&=\|\ft\|_D +\sum_{i=1}^s \|\ft\|_{J_i}\\
        &=\|\phi\|_D +\sum_{i=1}^s \|\ft\|_{J_i}\\
        &\le \|\ft+\gt\|_D +s+1+ \Big(\frac5\alpha+2\Big)\sum_{i=1}^s\|\ft+\gt\|_{J_i} - s(2\alpha+8)\\
        &\le\Big(\frac5\alpha+2\Big)\|\ft+\gt\|_{I^+} -7s-2\alpha s +1.
      \end{align*}
      By the symmetry condition \eqref{E:3.6} we also obtain that 
      $$\|\ft\|_{I^-}\le \Big(\frac5\alpha+2\Big)\|\ft+\gt\|_{I^+} -7s-2\alpha s +1.$$
      Adding these two inequalities yields our Main Claim since $s \ge 1$.
           \end{proof}
   \begin{thm}\label{T:3.8}
   Let $h\in L_1[0,1]$, with $\supp_H(h)\subset \suc([0,1])$, and let $0<\vp<1$, $0<\alpha\le 1$ and $b\in \R^+$.
   Assume that $\cS\subset \cD$, is such that
   \begin{align*}
   |c_I(h)|\ge\alpha b,\text{ if $I\in \cS$}, \text{ and }|c_I(h)|\le b,\text{ if $I\not\in \cS$.}
    \end{align*}  
    Then 
    $$\big\|P_{\cS_\vp} (h)\big\|\le  \frac{42}{\alpha^2\vp}  \|h\|.$$
   \end{thm}
      \begin{proof} After rescaling we can assume that $b=1$. Put
   $f= P_{S_\vp}(h)$ and $g=h-P_{\cS_\vp}(h)$.
   We note that $f$ and $g$ satisfy the assumptions of Lemma \ref{L:3.3} with 
   $$\cF=\big\{ I\kin \cD:  I\notin\cS_\vp,\text{  but }I^+\kin \cS_\vp  \text{ or } I^-\kin \cS_\vp\big\}.$$
   Indeed, condition (a) of  Lemma \ref{L:3.3} is clearly satisfied, and in order to verify (b) let $I\in \cF$. Without loss of generality we can assume that $I^+\in \cS_\vp$.
   Thus there is a $J\in\cS$, with $J\subset I$, and so that $J$ is maximal with that property. It follows therefore from the definition of $\cS_\vp$ that 
   $[J,I^+]\subset \cS_\vp$, and thus $[J,I]\cap \cF=\{I\}$, $|c_J(f)|=|c_J(h)|\ge \alpha$,
   and 
   $$|c_I(g)-c_J(f)|=|c_I(h)-c_J(h)|\ge \vp |c_J(f)|.$$
     Lemma \ref{L:3.3} yields that $$\|f+g\|\ge \frac{\alpha \vp}6|\cF|.$$  Setting
   $$\gb=g-\sum_{I\in \cF} c_I(g) h_I,$$
    then, by our assumption on  $h$,
     \begin{align*}
     \| f+ \gb\|&\le \|f+g\| +\|\gb-g\|\\
     &\le  \|f+g\| +\Big\|\sum_{I\in\cF} c_I(h) h_I\Big\|\\
    & \le   \|f+ g\| +|\cF|\\
    & \le\Big(1+\frac6{\alpha\vp} \Big) \|f+g\|.
     \end{align*}
     Note that since
   $\cF=\cF(f)$ (where $\cF(f)$ was  defined in \eqref{E:3.4c})
    the pair $f$ and $\gb$ satisfies the assumption of Lemma \ref{L:3.5} and we deduce that
    \begin{align*}
    \|h\|&=\|f+g\|\\
          &\ge \frac{\alpha\vp}{\alpha\vp+ 6}\|f+ \gb\|\\
          &\ge \frac{\alpha\vp}{\alpha\vp+ 6} \frac{\alpha}{\alpha+5}\|f\|\\
          &\ge \frac{\alpha^2\vp}{42}\|f\|= \frac{\alpha^2\vp}{42}\|P_{\cS_\vp}(h)\|
          \end{align*}
          which implies our claim.
   \end{proof}
   
   \begin{cor}\label{C:3.9} Let $f\in L_1[0,1]$, with $\supp_H(f)\subset \suc([0,1])$, $\cA\subset \cD$, $0<\vp<1$, $\rho\in \R^+$.
  Put $\cB=\cA\cap\{ I\in \cD: \rho<|c_I(f)|\le 2\rho\}$.
 
 Then there exists $\cC\subset \cD$, with $\cB\subseteq \cC\subseteq \cB_\vp(f)$, so that
 $$\|P_{\cC}(f)\|\le \frac{45738}{\vp} \|f\|.$$  
   \end{cor}
   \begin{proof} We first apply Theorem \ref{T:3.8} to the set $\cS=\{ J \in \cD: |c_J(f)|> 3\rho\}$, the numbers $b=3\rho$, $\alpha=1$ and $\vp=\frac13$.
   It follows that 
   \begin{equation}\label{E:3.9.1}
   \big\|P_{\cS_\vp}(f)\big\|\le 120\|f\|.
   \end{equation}
Note that 
   \begin{align*}
   \cS_\vp&    =\left\{ J\in\cD: \begin{matrix} \exists I\kin \cS, I\subset J\,\forall K\in [I,J]\\
                                                                         |c_I(f)-c_K(f)|\le \frac13 |c_I(f)| \end{matrix}\right\}
                 \subset  \big\{ J\in\cD:        |c_J(f)|            >2\rho\big\}        
     \end{align*}     
    Put  $\cB^{(1)}: =\cD\setminus \cS_\vp$,  and $g= P_{\cB^{(1)}}(f)$ then,
        \begin{equation}\label{E:3.9.2} \|g\|\le 121 \|f\|
            \end{equation}        and
      \begin{equation}\label{E:3.9.3}
      \big\{ J\in \cD: |c_J(f)|\le 2\rho\big\}\subseteq \cB^{(1)}
      \subseteq  \big\{ J\in \cD: |c_J(f)|\le 3\rho\big\}.
      \end{equation}
      Then we apply Theorem \ref{T:3.8} again, namely to the function $g$, the set 
      $$\cB^{(2)}=\big\{I\in \cD   : I\in\cA, \, \rho< |c_I(g)|\le 2\rho  \big\} ,$$
    and the numbers $b=3\rho $, $\alpha=\frac13$. We deduce that for each $\vp\in(0,1))$
   \begin{equation}\label{3.9.4}
   \big\|P_{\cB^{(2)}_\vp}(g)\big\|\le \frac{378}\vp \|g\|.
  \end{equation}
   Here we mean by $\cB^{(2)}_\vp$, to be precise, the set  $\cB^{(2)}_\vp(g)$.  
   Since for every $I\in\cD$, with $c_I(g)\not =0$, it follows that $c_I(g)=c_I(f) $, we deduce that
      \begin{align*}
   \cB^{(2)}_\vp(g)&=\left\{ J\in\cD: \begin{matrix} \exists I\kin \cA,I\ksubset J\quad\rho< |c_I(g)|\le 2\rho \text{\ s.th.} \\
                          \forall K\in [I,J]   |c_I(f)-c_K(f)|\le \vp |c_I(f)| \end{matrix}\right\}\\
                & \subseteq      \left\{ J\in\cD: \begin{matrix} \exists I\kin \cA, I\ksubset J\quad\rho< |c_I(f)|\le 2\rho \text{\ s.th.} \\
                          \forall K\in [I,J]   |c_I(f)-c_K(f)|\le \vp |c_I(f)| \end{matrix}\right\}= \cB^{(2)}_\vp(f).
   \end{align*} 
   Letting therefore $\cC=   \cB^{(2)}_\vp(y)$,  we deduce our claim from \eqref{E:3.9.2}, \eqref{3.9.4} and the fact that 
   $\cB\subset \cB^{(1)}$.
   \end{proof}
   We are now in the position to prove Theorem \ref{T:2.2}.
   \begin{proof}[Proof of Theorem \ref{T:2.2}] Let $f\in L_1$, and $\vp,\delta>0$.  We can assume that  $\vp<\frac13$ and that 
   $\supp_H(f)\subset \suc([0,1])$,  with  $|c_I(f)|\le 1$, for all $I\in\supp_H(f)$.
   We choose $m_0\in\N$, 
   so that $2^{-m_0}< \delta\le 2^{1-m_0}$, which implies that $m_0\le \log_2(2/\delta)$.
 
  For each $m=1,2,\ldots m_0$, we apply Corollary \ref{C:3.9} to  the function $f$, $\rho=2^{-m}$,  the set $\cA\cap\{ I\in\cD: \delta< |c_H(f)|\}$.
 We put $C_\vp= 45738/\vp$ and  
  $$\cB^{(m)}=\cA\cap \big\{ I\in\cD: 2^{m}\vee \delta<|c_I(f) |\le 2^{1-m}\big\}\text{ for $m=1,2\,\ldots m_0$}$$
  and deduce  that there are sets $\cC_m$, $\cB^{(m)}\subseteq \cC^{(m)}\subseteq\cB^{(m)}_\vp(f)$, so that
    \begin{equation}\label{E:2.2.1}
  \big\|P_{\cC^{(m)}}(f)\big\|\le C_\vp \|f\|\text{ for $m=1,2\,\ldots m_0$.}
  \end{equation}
   Since $\vp\le \frac13$ it follows for $i,j\in\{1,2,3\ldots m_0\}$, with $|i-j|\ge 2$, that $\cB^{(i)}_\vp(f)\cap \cB^{(j)}_\vp(f)=\emptyset$, and, thus, that 
   $\cC^{(i)}\cap \cC^{(j)}=\emptyset$.
   
   We let $\cF=\bigcup_{m=1, m \text{ odd}}^{m_0}\cC^{(m)}$. It follows  from \eqref{E:2.2.1} that 
   \begin{equation}\label{E:2.2.2}
   \big\|\cP_{\cF}(f)\big\|\sum_{m=1,\, m\text{ odd} }^{m_0} \|P_{\cC^{(m)}}(f)\|\le C_\vp \Big\lceil \frac{m_0}2\Big\rceil \le C_\vp\log\Big(\frac1\delta\Big). 
   \end{equation} 
  We are now applying again Corollary  \ref{C:3.9} to the function $g=f-\cP_{\cF}(f)$ and the set $\tilde\cA=(\cA\cap\{ I\in\cD: \delta< |c_H(f)|\})\setminus \cF$,
  and find sets $\tilde\cC^{(j)}$,  with  $\tilde\cB^{(j)}\subset \tilde\cC^{(j)}\subset \tilde\cB^{(j)}_\vp(g)$, 
  where 
   $$\tilde\cB^{(m)}=\tilde\cA\cap \big\{ I\in\cD: 2^{m}\vee \delta<|c_I(f) |\le 2^{1-m}\big\}\text{ for $m=1,2\,\ldots m_0$},$$
  so that 
  \begin{equation}\label{E:2.2.3}
  \big\|P_{\tilde\cC^{(m)}}(g)\big\|\le C_\vp \|g\|\text{ for $m=1,2\,\ldots m_0$.}
  \end{equation}
  We note that for every odd $m$ in $\{1,2\ldots m_0\}$  the set $\tilde\cC^{(m)}$ is empty and that therefore the 
 $\tilde\cC^{(m)}$ 's are pairwise disjoint. We also note that 
  $\tilde\cB_\vp^{(m)}(g)\cap \cF=\emptyset$  (since   $\tilde\cB_\vp^{(m)}(g) \subseteq\supp_H(g)$ which is disjoint from $\cF$) and thus that 
  $\tilde\cC^{(m)}\cap \cF=\emptyset$, for all $m=1,2\ldots m_0$. 
 Putting now $\cE= \cF\cup\bigcup_{m=1, m\text{ even}}^{m_0}\tilde C^{(m)}$ we obtain
 \begin{align*}
 \big\|P_\cE(f)\big\|&\le  \big\|P_{\cF}(f) \big\|+ \big\|P_\cE\big( f- P_{\cF}(f)\big)\big\|\\
                          &   \le \big\|P_{\cF}(f) \big\| +\sum_{m=1}^{m_0}  \big\|P_{\tilde \cC_m}(g)\big\|\\
                           &\le  C_\vp\log\Big(\frac1\delta\Big)+\Big\lfloor\frac{m_0}2 \Big\rfloor C_\vp \|g\|\\
                           &\le  C_\vp\log\Big(\frac1\delta\Big) + C_\vp \log\Big(\frac1\delta\Big) \Big(1+C_\vp\log\Big(\frac1\delta\Big)\Big)
\end{align*}
   which proves our claim.
     \end{proof}
   Our next example provides a lower bound for the constant on the right side of \eqref{E:2.2.10}.

   \begin{ex} For $n\in\N$ and $\delta=2^{-2n}$ we claim that there is a function $f\in L_1$ and an $\cA\subset \supp_H(f)$, so that for any $0<\vp<1$
   it follows that $\cA_\vp(f)=\cA$ and
   $$\|\cP_\vp(f)\|\ge \log\Big(\frac1\delta\Big)\|f\|.$$
   Indeed, we define $h_0=1_{[0,1]}$, and for $k\in\N$ and $j=1,2\ldots, 2^{k-1}$ we
   put
   $$h^{(k)}_j=2^{k-1} \big(1_{[(2j-2)2^{-k}, (2j-1) 2^{-k}) } -1_{[(2j-1) 2^{-k}, 2j2^{-k})}\big) .$$
   We observe that for any $n\in\N$ 
   \begin{align*} 
   \Big\|h_0+\sum_{k=1}^{2n} h^{(k)}_1\Big\|=1
   \intertext{ and for some universal constant $c>0$.}
    \Big\|h_0+\sum_{k=1}^{2n} h^{(2k)}_1\Big\|>cn.
   \end{align*}
   We secondly observe that the joint distribution of  the sequence 
  $$ h_0, \frac12(h^{(2)}_1+h^{(2)}_2),\frac14(h^{(4)}_1+h^{(4)}_2+h^{(4)}_3+h^{(4)}_4), 
   \frac18(h^{(6)}_1+h^{(6)}_2+\ldots +h^{(6)}_8), \ldots,$$
   is equal to the joint distribution of $h_0$, $h^{(1)}_1$,$h^{(2)}_1,\ldots$
   
   It follows therefore that 
   \begin{align*}
   &\Big\|h_0 + \sum_{k=1}^{2n} 2^{-k} \sum_{j=1}^{2^{k}} h^{2k}_j\Big\|= \Big\|h_0+\sum_{k=1}^{2n} h^{(k)}_1\Big\| = 1, \text{ and }\\
   &\Big\|h_0 + \sum_{k=1}^{n} 2^{-2k} \sum_{j=1}^{2^{2k}} h^{4k}_j\Big\|\ge cn.
   \end{align*}
   Therefore if we choose 
   $$f=h_0 + \sum_{k=1}^{2n} 2^{-k} \sum_{j=1}^{2^{k}} h^{2k}_j$$
   and $\cA=\{ h^{(0)}\} \cup \{h^{4k}_j: k=1,2\ldots n\text{ and } j=1,2\ldots 2^{2k}\}$,
   we obtain for $\delta =2^{-2n}$ , and any $0<\vp<1$ that $\cA_\vp(f)=\cA$ and $\|P_{\cA}(f)\|\ge cn\sim \log(1/\delta)$.
   \end{ex}

\end{document}